\documentclass[11pt,a4paper]{amsart}
\usepackage{amsmath,amssymb,latexsym,amsthm,enumerate}
\usepackage{xcolor} %Colors in text
\usepackage{graphicx} %\includegraphics for .ps, .eps (dvips); .pdf, .jpg, .png (pdflatex)

\usepackage{hyperref} %To be able to click on the links in the paper

\theoremstyle{plain}
\newtheorem{theorem}{Theorem}[section]
\newtheorem{proposition}[theorem]{Proposition}
\newtheorem{lemma}[theorem]{Lemma}

\newtheorem{problem}{Problem}
\theoremstyle{definition}
\newtheorem{definition}[theorem]{Definition}
\newtheorem{example}[theorem]{Example}
\newtheorem{remark}[theorem]{Remark}

\theoremstyle{remark}

\numberwithin{equation}{section}

\newcounter{numpar}[section]

\newcounter{numitem}[section]

\newcommand*{\ol}{\overline}

\newcommand*{\dd}{d}     %For dx, dy, etc.
\newcommand*{\eps}{\varepsilon}
\newcommand*{\ka}{\varkappa}

\newcommand*{\cA}{\mathcal A}
\newcommand*{\cB}{\mathcal B}

\newcommand*{\cE}{\mathcal E}
\newcommand*{\cF}{\mathcal F}
\newcommand*{\cG}{\mathcal G}

\newcommand*{\bbN}{\mathbb N}

\newcommand*{\bbR}{\mathbb R}

\newcommand*{\EE}{\mathsf E}
\newcommand*{\PP}{\mathsf P}
\newcommand*{\PR}{\mathsf{Pr}}
\newcommand*{\QQ}{\mathsf Q}

\newcommand*{\loc}{{\mathrm{loc}}}

\DeclareMathOperator{\Law}{Law}

\begin{document}
\title[Processes embedded in GBM]{Processes
that can be embedded\\in a
geometric Brownian motion}

\author{Alexander Gushchin}
\address{Steklov Mathematical Institute
\and
Higher School of Economics,
Moscow, Russia}
\email{gushchin@mi.ras.ru}

\author{Mikhail Urusov}
\address{University of Duisburg-Essen,
Essen, Germany}
\email{mikhail.urusov@uni-due.de}

%\thanks will become a 1st page footnote.
%\thanks{}


%\subjclass[2010]{60H10; 60J60; 60J55}

\begin{abstract}
The main result is a counterpart
of the theorem of Monroe
[\emph{Ann. Probability} \textbf{6} (1978) 42--56]
for a geometric Brownian motion:
A process is equivalent to a time change
of a geometric Brownian motion
if and only if it is a nonnegative supermartingale.
We also provide a link between
our main result and Monroe
[\emph{Ann. Math. Statist.} \textbf{43} (1972) 1293--1311].
This is based on the concept
of a \emph{minimal} stopping time,
which is characterised in Monroe
[\emph{Ann. Math. Statist.} \textbf{43} (1972) 1293--1311]
and Cox and Hobson
[\emph{Probab. Theory Related Fields} \textbf{135} (2006) 395--414]
in the Brownian case.
We finally suggest a sufficient condition for minimality
(for the processes other than a Brownian motion)
complementing the discussion in the aforementioned
papers.
\end{abstract}

\maketitle

%\tableofcontents

%=======
\section{Introduction and Main Result}
\label{sec:i}
In his seminal paper
Monroe~\cite{Monroe:78}
proves that a c\`adl\`ag process
is equivalent to a finite time change
of a Brownian motion
if and only if
it is a semimartingale.
Here,
the processes are said to be equivalent
if they have the same law.
We prove a counterpart
of this result for a geometric
Brownian motion:

\begin{theorem}
\label{th:i1}
(i) Let $X=(X_s)_{s\geq0}$
be a nonnegative supermartingale with $\EE X_0\leq1$.
Then there exists a filtered probability space
$(\Omega,\cF,(\cF_t)_{t\geq0},\PP)$,
an $(\cF_t,\PP)$-Brownian motion $W=(W_t)$
and a $[0,\infty]$-valued
$(\cF_t)$-time change $(T_s)$
such that the processes
$(X_s)_{s\geq0}$ and $(Z_{T_s})_{s\geq0}$
have the same law,
where $Z_t=e^{W_t-t/2}$, $t\geq0$.

(ii) Conversely, for any $[0,\infty]$-valued
$(\cF_t)$-time change $(T_s)$,
the process $(Z_{T_s})$ is a nonnegative
$(\cF_{T_s},\PP)$-supermartingale.
\end{theorem}

Part~(ii) immediately follows from
the optional sampling theorem for
nonnegative supermartingales
applied to $(Z_t)$,
so the task is to prove part~(i).

We follow the usual convention
of working with c\`adl\`ag processes.
In particular, ``supermartingale'' means
``c\`adl\`ag supermartingale''.
Let us recall that a \emph{time change}
is a family $(T_s)_{s\geq0}$ of
stopping times such that the
maps $s\mapsto T_s$
are a.s. nondecreasing and right-continuous.
In contrast to Monroe's~\cite{Monroe:78} setting
the stopping times $T_s$ need not be finite here.
This is natural in our setting because
the nonnegative martingale
$(Z_t)$ has limit $Z_\infty\equiv0$
and is closed as a supermartingale by this
limit.\footnote{An
immediate consequence of Theorem~\ref{th:i1} is
the statement obtained from Theorem~\ref{th:i1}
by replacement of
``nonnegative''
with
``strictly positive''
and ``$[0,\infty]$-valued''
with
``finite''.}

On the one hand,
it is often helpful to know
whether a random process
can be considered as a
time-changed process
with a simple structure.
In finance,
the modelling approach based
on time changes was, in fact,
even inspired by Monroe's theorem,
see~\cite{AneGeman:00}.
Nowadays this modelling approach
is very popular in financial mathematics,
see~\cite{Barndorff-NielsenShiryaev:10}
and the references therein.
On the other hand, Monroe's theorem
is one of the offsprings
of the Skorokhod Embedding Problem
(abbreviated below as the SEP).
The latter was originally formulated and
solved in~\cite{Skorokhod:61}
(English translation in~\cite{Skorokhod:65})
and gave rise to a huge amount of literature.
In~\cite{Obloj:04} one finds a comprehensive
survey of the state of the art to~2004,
in particular, more than twenty different
approaches to solve the SEP
with the relations between them,
different settings and generalisations
as well as some other offsprings.
Skorokhod's motivation for the SEP
was proving limit theorems
(e.g. one can obtain the law of the iterated logarithm
for random walks from that for a Brownian motion),
but in recent years there appeared
also other applications.
The methodology based on
Skorokhod embedding
and pathwise inequalities proved to be
important for finding model-independent
bounds for option prices and
robust hedging strategies.\footnote{Let
us note that robust hedging methods
may often outperform classical
in-model hedging when there is
model ambiguity and/or market frictions,
see~\cite{OblojUlmer:12}.}
That gave rise to further research in this
direction, which continues nowadays, see e.g.
\cite{Hobson:98a},
\cite{BrownHobsonRogers:01},
\cite{CoxHobsonObloj:08},
\cite{CoxObloj:11a},
\cite{CoxObloj:11},
\cite{AcciaioBeiglbockPenknerSchachermayerTemme:13},
\cite{BeiglbockHenry-LaborderePenkner:13},
\cite{CoxWang:13},~\cite{DolinskySoner:14}.
One finds more details
and many further references
in recent surveys on the SEP
and its applications to
robust pricing and hedging
\cite{Hobson:11} and~\cite{Obloj:12}.

In spite of the generality of Monroe's theorem,
we cannot obtain Theorem~\ref{th:i1}
as its consequence
(the reason is described in Section~\ref{sec:app})
and should therefore
prove Theorem~\ref{th:i1} independently.
To prove it
we proceed similarly to Monroe~\cite{Monroe:78},
although some technical details
are elaborated differently,
which is due to natural
differences between the settings.
In the first step, we consider the SEP
for a geometric Brownian motion
(i.e. embedding of a single random variable
in a geometric Brownian motion).
Different solutions to this problem
(in fact, to the one for a Brownian motion
with drift) were proposed in \cite{Hall:69},
\cite{GranditsFalkner:00},
\cite{Peskir:00},
\cite{AnkirchnerHeyneImkeller:08},
\cite{AnkirchnerStrack:11},
and~\cite{AnkirchnerHobsonStrack:14}.
In Section~\ref{par:prv},
we suggest an alternative construction,
which is, in our view,
of interest in its own right.
Also this construction
will be convenient in Section~\ref{par:pdt}.
In the literature, there are already explicit
embeddings in time-homogeneous diffusions,
see~\cite{PedersenPeskir:01},
\cite{CoxHobson:04},
\cite{AnkirchnerHobsonStrack:14},
Section~9 in~\cite{Obloj:04},
Section~4.3 in~\cite{Hobson:11}
and the references therein.
However, to the best of our knowledge,
the construction that we present
in Section~\ref{par:prv}
did not appear in the papers on the subject,
although the ideas behind it
are of course present in the literature.
Most notably,
our construction can be viewed as a rework
of the original Skorokhod's construction
(see \cite{Skorokhod:65}
or Section~3.12 in~\cite{Hobson:11})
for the case of a geometric Brownian motion.
In the second step, we embed
discrete-time supermartingales
in a geometric Brownian motion.
Typically, if it is known how to embed
one random variable,
there is no problem to embed
a discrete-time process.
However, in our situation,
this step turns out
to be surprisingly technical.
The reason is that the time change
is allowed to take infinite value,
see Section~\ref{par:pdt} for more details.
In the third step, we justify a passage
to the continuous-time limit.
This part is closer to the
corresponding part of the proof
of Monroe's theorem,
and we, in fact, just refer
to~\cite{Monroe:78} at some point
(see Section~\ref{par:pct}).
In Section~\ref{sec:app},
we discuss some issues
related to \emph{minimal stopping times}
for a Brownian motion,
the concept studied in another
Monroe's paper~\cite{Monroe:72}
and taken on in many subsequent works
on the SEP and its offsprings.
In particular, we explain
why Theorem~\ref{th:i1}
is not a consequence of~\cite{Monroe:72}
and~\cite{Monroe:78}
and provide in Theorem~\ref{th:app4}
an equivalent formulation
of Theorem~\ref{th:i1},
which complements
the discussion in~\cite{Monroe:72}.
For a Brownian motion,
minimality is characterised in~\cite{Monroe:72}
and, in a more general situation,
in~\cite{CoxHobson:06}.
In Section~\ref{sec:mst},
we study minimal stopping times
for other processes.
Namely, in Theorem~\ref{th:mst1},
we give a sufficient condition for minimality,
which is new
and complements the discussion of minimal
stopping times for processes other than
a Brownian motion
in Section~8 in~\cite{Obloj:04},
Section~3.4 in~\cite{Hobson:11}
and Section~2.2 in~\cite{Obloj:12}.
We will see that Theorem~\ref{th:mst1}
applies in many specific situations.

Let us finish the introduction by discussing
the embedding in the process
$\ol Z^{a,b}_t=e^{aW_t+bt}$, $t\geq0$,
where $a\ne0$, $b\in\bbR$.
First let $b\ne0$, hence the
(possibly infinite) limit
$\ol Z^{a,b}_\infty
:=\lim_{t\to\infty}\ol Z^{a,b}_t$
is well-defined, i.e. it is natural to
consider $[0,\infty]$-valued time changes.
Then Theorem~\ref{th:i1} implies that
a c\`adl\`ag process $\ol X$
is equivalent to a $[0,\infty]$-valued
time change of $\ol Z^{a,b}$
if and only if $\ol X^\lambda$
is a nonnegative supermartingale with
$\EE\ol X_0^\lambda\leq1$,
where $\lambda=-\frac{2b}{a^2}$.
Note that if $b>0$,
then $\ol X$ is allowed to take value $+\infty$.
Let now $b=0$.
Since $\lim_{t\to\infty}\ol Z^{a,0}_t$
does not exist,
it is now natural to consider
only finite time changes.
Then Monroe's theorem implies that
a c\`adl\`ag process $\ol X$
is equivalent to a finite time change
of $\ol Z^{a,0}$ if and only if
it is a strictly positive semimartingale.

%=======
\section{Proof of Theorem~\protect\ref{th:i1}}
\label{sec:p}
\subsection{Embedding of a Single Random Variable}
\label{par:prv}
We will use the notation $\mu_W$
for the Wiener measure on
$(C(\bbR_+),\cB(C(\bbR_+)))$
and $\mu_L$ for the Lebesgue measure
on $([0,1],\cB([0,1]))$.
For some random variables $\xi$ and $\eta$,
we write $\xi\sim\eta$ to express that
$\xi$ and $\eta$ have the same law.

\begin{lemma}
\label{lem:p1}
Let $\xi$
be a nonnegative random variable with $\EE \xi\leq1$. Consider
the filtered probability space
$(\Omega,\cF,(\cF_t)_{t\geq0},\PP)$ with
$$
\Omega=C(\bbR_+)\times [0,1],
\quad
\cF=\cB(C(\bbR_+))\otimes \cB([0,1]),
\quad
\PP=\mu_W\times\mu_L,
$$
and
$\cF_t = \bigcap_{\eps>0}\sigma(R,B_s;s\in[0,t+\eps])$,
where
the random variable $R$
and
the process $B=(B_t)$
on $\Omega$
are defined as follows:
for $\omega=(x,r)$,
$R(\omega):=r$, $B_t(\omega):=x(t)$.
In particular,
$R$ is $\cF_0$-measurable
and uniformly distributed on $[0,1]$,
and $B$ is an $(\cF_t,\PP)$-Brownian motion.
Then there exists a $[0,\infty]$-valued
$(\cF_t)$-stopping time $\tau$
such that $\xi\sim Y_\tau$,
where $Y_t=e^{B_t-t/2}$, $t\geq0$.
\end{lemma}

Let us remark at this point
that the converse, obviously, holds as well,
i.e. a random variable
can be embedded in
a geometric Brownian motion
if and only if
it is nonnegative
and its expectation is less than
or equal to one.

\begin{proof}
Let $F$ denote
the distribution function of $\xi$
and let the quantile function
$F^{-1}\colon[0,1]\to[0,+\infty]$
be defined as the right-continuous inverse of $F$,
i.e. $F^{-1}(r)=\inf\{x\in\bbR_+\colon F(x)>r\}$;
here and below $\inf\emptyset=+\infty$.
It is well known that $F^{-1}(R)$
has the same distribution as $\xi$.
Therefore,
below we assume
without loss of generality
that $\xi=F^{-1}(R)$.

Let us set $h(r)=\int_0^r F^{-1}(s)\,\dd s$,
$g(r)=r-h(r)$, $r\in[0,1]$.
Then $h$ is a nondecreasing convex function on $[0,1]$,
$h(0)=0$, $h(1)=\EE\xi\leq1$.
If $h(1)<1$, the equation
$g(x)=c$ for $0\leq c<g(1)$
has exactly one solution, say,
$\theta=\theta(c)\in[0,1]$.
For such $\theta$,
we put $U(\theta)=F^{-1}(\theta)$,
$V(\theta)=+\infty$.
If $g(1)\leq c < g^*:=\max_{r\in[0,1]}g(r)$,
the same equation has two solutions
$\theta_1<\theta_2$ in $[0,1]$.
For such $\theta_1$ and $\theta_2$,
we put $U(\theta_1)=U(\theta_2)=F^{-1}(\theta_1)$,
$V(\theta_1)=V(\theta_2)=F^{-1}(\theta_2)$.
For $\theta\in[0,1]$ such that $g(\theta)=g^*$,
we put $U(\theta)=V(\theta)=1$.
We thus defined the functions
$U\colon[0,1]\to[0,1]$
and $V\colon[0,1]\to[1,+\infty]$.
Finally, let us introduce the random variables
$\eta:=g(R)$, $\alpha:=U(R)$ and $\beta:=V(R)$.
Note that $\alpha$ and $\beta$ are, in fact,
functions of~$\eta$.
In Figure~\ref{fig:p1}
we explain the structure
of the random variables
$\xi$, $\alpha$ and $\beta$
via the graphs of the functions
$h$ and~$g$.

\begin{figure}[htb!]
\includegraphics[width=.49\textwidth]{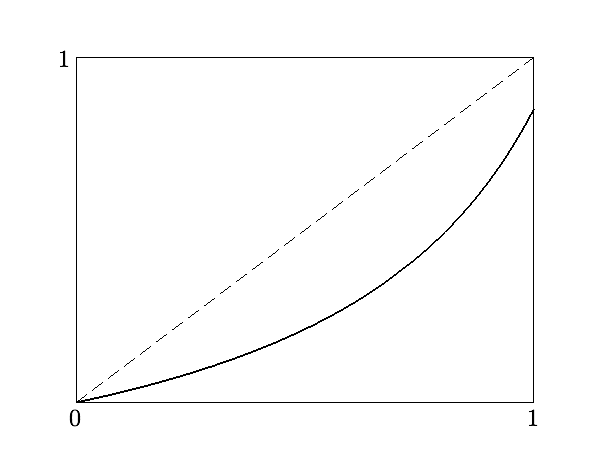}
\includegraphics[width=.49\textwidth]{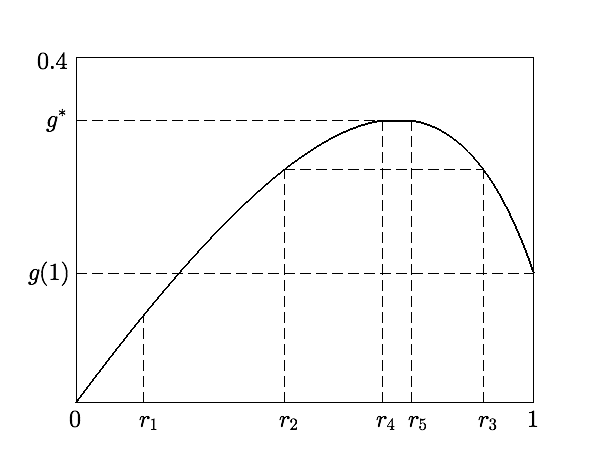}
\caption{In the figure on the left,
the solid line is the graph of the function
$r\mapsto h(r)$,
the dashed line is that of
the identity function $r\mapsto r$.
The relation with the structure
of the random variable $\xi$
is explained by the formula
$\xi=F^{-1}(R)=h'_+(R)$,
where $h'_+$
denotes the right derivative
of the convex function~$h$.
In the figure on the right,
the solid line is the graph
of the function $r\mapsto g(r)$.
The structure of the random variables
$\alpha$ and $\beta$
can be explained as follows:
if $R=r_1$, then $\alpha=F^{-1}(r_1)=\xi$
and $\beta=+\infty$;
if $R\in\{r_2,r_3\}$, then
$\alpha=F^{-1}(r_2)$ and
$\beta=F^{-1}(r_3)$;
if $R\in[r_4,r_5]$, then $\alpha=\beta=1$.}
\label{fig:p1}
\end{figure}

The key point of our construction
is the following characterisation
of the conditional law of $\xi$ given~$\eta$:

\smallskip
(a) A.s. on the event $\{\eta<g(1)\}$
it is concentrated on the one-point set $\{\xi\}$
(note that $\xi$ is a function of $\eta$ on this event
because $R$ and $\eta$ are in a one-to-one
correspondence on $\{\eta<g(1)\}$);

(b) A.s. on the event $\{\eta\geq g(1)\}$
it is concentrated on the set $\{\alpha,\beta\}$
and, moreover,
\begin{equation}
\label{eq:p1}
\EE(\xi|\eta)=1\text{ a.s. on }\{\eta\geq g(1)\},
\end{equation}
which determines the conditional law of $\xi$
given $\eta$ in a unique way.

\smallskip
To prove~\eqref{eq:p1},
it is sufficient to check that
for any interval $(a,b)\subset[g(1),g^*]$,
it follows
$\EE\xi 1_{\{\eta\in(a,b)\}}
=\PP\bigl(\eta\in(a,b)\bigr)$.
We have
$\{\eta\in(a,b)\}
=\{g(R)\in(a,b)\}
=\{R\in(r_0,r_1)\cup(r_2,r_3)\}$
with $g(r_0)=g(r_3)=a$
and $g(r_1)=g(r_2)=b$,
therefore,
\begin{align*}
\EE\xi 1_{\{\eta\in(a,b)\}}
&=\EE F^{-1}(R) 1_{\{g(R)\in(a,b)\}}
=(h(r_1)-h(r_0))+(h(r_3)-h(r_2))\\
&=(r_1-r_0)+(r_3-r_2)
=\PP\bigl(R\in(r_0,r_1)\cup(r_2,r_3)\bigr)
=\PP\bigl(\eta\in(a,b)\bigr)
\end{align*}
(recall the definitions of the functions $g$ and~$h$).
The other statements in (a) and~(b) above
are clear.

Now we define $\tau$ by the formula
$$
\tau=\inf\{t\in\bbR_+: Y_t\notin(\alpha,\beta)\}.
$$
Since the random variables $\alpha$
and $\beta$ are $\cF_0$-measurable,
$\tau$ is an $(\cF_t)$-stopping time.
Let us prove that the conditional law of $Y_\tau$
given $\eta$ admits the following characterisation:

\smallskip
(A) A.s. on the event $\{\eta<g(1)\}$
it is concentrated on the one-point set $\{\xi\}$;

(B) A.s. on the event $\{\eta\geq g(1)\}$
it is concentrated on the set $\{\alpha,\beta\}$
and, moreover,
\begin{equation}
\label{eq:p2}
\EE(Y_\tau|\eta)=1\text{ a.s. on }\{\eta\geq g(1)\}.
\end{equation}

\smallskip
Indeed, if $\PP(\eta<g(1))>0$,
then on $\{\eta<g(1)\}$
it holds $\beta=\infty$ and,
since $\lim_{t\to\infty} Y_t=0$ a.s.,
we have $Y_\tau=\alpha=F^{-1}(R)=\xi$
on this event
(the case $\alpha=0$, where $\tau=\infty$, is included).
The first statement in~(B) is clear.
It remains to check~\eqref{eq:p2}.
Let us take $r\in[0,1]$
such that $g(r)\geq g(1)$.
Since $V(r)<\infty$ for such $r$, the process
$\big(Y(\cdot,r)_{t\wedge\tau(\cdot,r)}\big)$
is a bounded martingale on $C(\bbR_+)$
with respect to the coordinate filtration
and the Wiener measure~$\mu_W$, i.e.
$\int_{C(\bbR_+)}Y(x,r)_{\tau(x,r)}\mu_W(dx)=1$,
which means that
$\EE(Y_\tau|R)=1$
a.s. on $\{\eta\geq g(1)\}$.
Statement~\eqref{eq:p2} now follows
by the tower property of conditional expectations.

Comparing (a),~(b) and (A),~(B) above
we obtain that the conditional laws
of $\xi$ and $Y_\tau$ given $\eta$ coincide.
Hence, their unconditional laws coincide.
This concludes the proof.
\end{proof}

\begin{remark}
\label{rem:p1}
(i) At first glance it might seem
tempting to prove Lemma~\ref{lem:p1}
via a construction like
Doob's construction
for embedding in a Brownian motion
(see the paragraph following
Problem~\ref{pr:app1}
in Section~\ref{sec:app}).
However, this does not work
because $Y$ is transient.
Namely, for the stopping time $\tau$
defined similarly to~\eqref{eq:app1},
we shall typically have
$Y_\tau\neq f(Y_1)\sim\xi$.

(ii) The proof of Lemma~\ref{lem:p1}
provides an alternative solution
to the Skorokhod embedding problem
for a geometric Brownian motion.
It is reminiscent of Hall's solution~\cite{Hall:69},
although qualitatively different from it.
Both in~\cite{Hall:69}
and in the proof above
the stopping time is constructed
as the hitting time of two levels,
$\alpha$ and $\beta$,
that are obtained via a randomization.
However, these randomizations are very different.
For instance,
if the law of $\xi$ has no atoms,
then $\beta$ in our construction
is always a deterministic function
of $\alpha$,
while in~\cite{Hall:69}
the random vector $(\alpha,\beta)$
is ``genuinely two-dimensional'',
i.e. the conditional distribution
of $\beta$ given $\alpha$
is nondegenerate.

(iii) The construction in the proof
of Lemma~\ref{lem:p1} appeared
earlier in statistical context
in~\cite{Birnbaum:61}
in the proof that each binary experiment
is equivalent to a mixture of strictly ordered
simple binary experiments.
Here we tailored the construction to our
situation and found a short proof via~\eqref{eq:p1}.

(iv) Other solutions to the SEP for a geometric Brownian motion or, equivalently, for a Brownian motion with drift were obtained in a number of papers mentioned in the introduction. One more way to construct required embeddings is to reduce the problem to the SEP for a Brownian motion and non-centred target distributions via change of time, see the explanation in Section~\ref{sec:app}. Then one can use the solution proposed by Cox~\cite{Cox:08} or just run the Brownian motion until it hits the mean of the distribution and then use a solution in the centred case.

(v) For the sequel, let us remark
that the levels $\alpha$ and $\beta$
in the proof of Lemma~\ref{lem:p1}
depend on an external
uniformly distributed on $[0,1]$
random variable $R$
and on the distribution $F$
we want to embed:
with a slight abuse of notation
we will write
$\alpha=\alpha(F,R)$,
$\beta=\beta(F,R)$
(recall Figure~\ref{fig:p1}
and observe that the functions
$h$ and $g$ are constructed via~$F$).
\end{remark}

%=======
\subsection{Embedding of a Discrete-Time Supermartingale.}
\label{par:pdt}
Using Lemma~\ref{lem:p1}
we now embed a nonnegative
discrete-time supermartingale
in a geometric Brownian motion.
Let
$\bbN_0:=\bbN\cup\{0\}$.
We will also use the notation
$\Law_\PP(\xi)$
(resp. $\Law_\PP(\xi|\eta)$
\emph{or} $\Law_\PP(\xi|\cG)$)
for the law of $\xi$ under~$\PP$
(resp. for the conditional law of $\xi$
given $\eta$ \emph{or}
given $\cG$ under~$\PP$)
whenever $\xi$ and $\eta$
are random elements
and $\cG$ is a sub-$\sigma$-field
on some probability space
$(\Omega,\cF,\PP)$.

\begin{lemma}
\label{lem:p2}
Let $X=(X_n)_{n\in\bbN_0}$
be a nonnegative supermartingale
with $\EE X_0\leq1$.
Consider
the filtered probability space
$(\Omega,\cF,(\cF_t)_{t\geq0},\PP)$ with
$$
\Omega=C(\bbR_+)\times [0,1]^{\bbN_0},
\quad
\cF=\cB(C(\bbR_+))\otimes \cB([0,1])^{\otimes\bbN_0},
\quad
\PP=\mu_W\times\mu_L^{\bbN_0},
$$
and
$\cF_t = \bigcap_{\eps>0}\sigma
(R_n,B_s;n\in\bbN_0,s\in[0,t+\eps])$,
where
the random variables $R_n$
and
the process $B=(B_t)$
on $\Omega$
are defined as follows:
for $\omega=(x,r_0,r_1,\ldots)$,
$R_n(\omega):=r_n$, $B_t(\omega):=x(t)$.
In particular,
$R_n$, $n\in\bbN_0$,
are independent $\cF_0$-measurable
uniformly distributed on $[0,1]$ random variables,
and $B$ is an $(\cF_t,\PP)$-Brownian motion
(note that independence of
$B$ and $(R_n)$ is included
in this statement).
Then there exists
a nondecreasing family $(\tau_n)$
of $[0,\infty]$-valued
$(\cF_t)$-stopping times
such that the processes
$(X_n)$ and $(Y_{\tau_n})$ have the same law,
where $Y_t=e^{B_t-t/2}$, $t\geq0$.
\end{lemma}

\begin{proof}
Let $F_0$ denote the distribution function of $X_0$.
Consider the $(\cF_t)$-stopping time
$$
\tau_0=\inf\big\{t\in\bbR_+:
Y_t\notin\big(\alpha(F_0,R_0),\beta(F_0,R_0)\big)\big\},
$$
where the notations $\alpha(F_0,R_0)$
and $\beta(F_0,R_0)$
are introduced in Remark~\ref{rem:p1}~(v).
By Lemma~\ref{lem:p1},
the random variables $X_0$
and $Y_{\tau_0}$
have the same law.
For the sequel let us also observe
that the family $(R_n)_{n\geq1}$
is independent of
$\sigma(\tau_0,Y_t;\,t\geq0)$
under~$\PP$.

We proceed by induction.
The induction hypothesis is as follows.
For some $k\in\bbN_0$,
we constructed a nondecreasing
family $(\tau_n)_{0\leq n\leq k}$
of $(\cF_t)$-stopping times
such that
\begin{equation}
\label{eq:p3}
\Law_\PR(X_0,\ldots,X_k)
=\Law_\PP(Y_{\tau_0},\ldots,Y_{\tau_k})
\end{equation}
(here and below $\PR$
denotes the probability measure
on the space,
where the sequence $(X_n)$ is defined)
and that
\begin{equation}
\label{eq:p4}
(R_n)_{n\geq k+1}
\text{ is independent of }
\sigma\big(Y_t,\tau_n;\,t\geq0,0\leq n\leq k\big)
\text{ under }
\PP.
\end{equation}
We need to construct
an $(\cF_t)$-stopping time
$\tau_{k+1}\geq\tau_k$
such that \eqref{eq:p3}
and~\eqref{eq:p4}
hold with $k$
replaced by~$k+1$.

In what follows we will work
with the random variables like
$\frac{X_{k+1}}{X_k}$
employing the convention
$\frac00:=1$
(note that $X_{k+1}=0$
$\PR$-a.s. on the set $\{X_k=0\}$
because $(X_n)$ is a nonnegative
supermartingale).
Let us remark that
$\EE_\PR\big(\frac{X_{k+1}}{X_k}\big|
X_0,\ldots,X_k\big)\leq1$ $\PR$-a.s.
The idea is now to embed
$\Law_\PR\big(\frac{X_{k+1}}{X_k}\big|
X_0,\ldots,X_k\big)$
via Lemma~\ref{lem:p1}
in the geometric Brownian motion
$({Y_{t+\tau_k}}/{Y_{\tau_k}})$,
but this requires some additional technical
work because $\tau_k$
may take infinite value.

Let us consider the regular
conditional distribution function
$F_{k+1}=(F_{k+1}(x|x_0,\ldots,
x_k))_{x,x_0,\ldots,x_k\in\bbR_+}$
for the random variable
$\frac{X_{k+1}}{X_k}$
given $X_0,\ldots,X_k$.
Namely, for each
$x_0,\ldots,x_k\in\bbR_+$,
$F_{k+1}(\cdot|x_0,\ldots,x_k)$
is a distribution function
of a probability measure on $\bbR_+$;
for each $x\in\bbR_+$,
$F_{k+1}(x|\cdot)$ is a Borel function
on $\bbR_+^{k+1}$,
and the random variable
$F_{k+1}(x|X_0,\ldots,X_k)$
is a version of the conditional probability
$\PR\big(\frac{X_{k+1}}{X_k}\leq x\big|
X_0,\ldots,X_k\big)$.
We define $\tau_{k+1}$
by the formula
$$
\tau_{k+1}=\tau_k+\inf\left\{t\in\bbR_+:
\frac{Y_{t+\tau_k}}{Y_{\tau_k}}\notin\big(
\alpha(F_{k+1}
(\cdot|Y_{\tau_0},\ldots,Y_{\tau_k}),R_{k+1}),
\beta(F_{k+1}
(\cdot|Y_{\tau_0},\ldots,Y_{\tau_k}),R_{k+1})
\big)\right\}
$$
($\tau_{k+1}:=\infty$ on the event $\{\tau_k=\infty\}$),
which is an $(\cF_t)$-stopping time
because $R_{k+1}$ is $\cF_0$-measurable
and $F_{k+1}(\cdot|Y_{\tau_0},\ldots,Y_{\tau_k})$
is known at time $\tau_k$.
Let us note that \eqref{eq:p4}
with $k$ replaced by $k+1$
follows from the formula for $\tau_{k+1}$,
\eqref{eq:p4} and the fact
that $R_{k+1},R_{k+2},\ldots$
are independent under~$\PP$.
It remains to prove that
\begin{equation}
\label{eq:p5}
\Law_\PR(X_0,\ldots,X_{k+1})
=\Law_\PP(Y_{\tau_0},\ldots,Y_{\tau_{k+1}}).
\end{equation}
If $\PP(\tau_k=\infty)=1$
(equivalently, $\PR(X_k=0)=1$),
then $Y_{\tau_{k+1}}=0$ $\PP$-a.s.
and $X_{k+1}=0$ $\PR$-a.s.,
so \eqref{eq:p5}
follows from~\eqref{eq:p3}.
Below we assume that $\PP(\tau_k<\infty)>0$.
Let us introduce the probability measure $\QQ$
on $(\Omega,\cF)$ by the formula
$$
\QQ(\cdot):=\PP(\cdot|\tau_k<\infty).
$$
We will use the notation
$\cG:=\sigma(Y_{\tau_0},\ldots,Y_{\tau_k})$.
One can easily check that,
for any nonnegative
random variable $Z$, we have
\begin{equation}
\label{eq:p6}
\EE_\QQ(Z|\cG)
=\EE_\PP(Z|\cG)
\quad\PP\text{-a.s. on }
\{\tau_k<\infty\}
\end{equation}
(note that $\PP$-a.s. we have
$\{\tau_k<\infty\}=\{Y_{\tau_k}>0\}\in\cG$)
or, equivalently,
\begin{equation}
\label{eq:p8}
\EE_\QQ(Z|\cG)
=\EE_\PP(Z|\cG)
\quad\QQ\text{-a.s.\phantom{ on }}
\phantom{\{\tau_k<\infty\}}
\end{equation}
In fact, the identities~\eqref{eq:p6}
and~\eqref{eq:p8}
hold even conditionally on~$\cF_{\tau_k}$.
It follows from~\eqref{eq:p8}
and~\eqref{eq:p4}
that, for~$x\in[0,1]$,
$\QQ$-a.s. we have
\begin{equation}
\label{eq:p9}
\QQ(R_{k+1}\leq x|\cG)
=\PP(R_{k+1}\leq x|\cG)
=\PP(R_{k+1}\leq x)=x,
\end{equation}
i.e. under $\QQ$
conditionally on $\cG$
the random variable $R_{k+1}$
is uniformly distributed on $[0,1]$.
Let now $A\in\cG$ and $B=\{R_{k+1}\leq x\}$.
Then, by~\eqref{eq:p9},
$$
\QQ(A\cap B)
=\EE_\QQ\left[1_A \QQ(B|\cG)\right]
=\EE_\QQ\left[1_A \QQ(B)\right]
=\QQ(A)\QQ(B),
$$
i.e. $\cG$ and $R_{k+1}$
are independent under~$\QQ$.
One can deduce from the
strong Markov property
of Brownian motion
(e.g.~in the
form~\cite[Ch.~III, Th.~3.1]{RevuzYor:99})
that under~$\QQ$ the process
$(Y_{t+\tau_k}/Y_{\tau_k})$
is a geometric Brownian motion
independent of $\cF_{\tau_k}$.
Since $\cG\subset\cF_{\tau_k}$
and $R_{k+1}$ is $\cF_{\tau_k}$-measurable
(even $\cF_0$-measurable),
we get
\begin{equation}
\label{eq:p11}
\left({Y_{t+\tau_k}}/{Y_{\tau_k}}\right)_{t\geq0},\;
\cG\text{ and }R_{k+1}
\text{ are independent under }\QQ.
\end{equation}
Summarising, we have:
\begin{enumerate}
\item
Under $\QQ$ conditionally on $\cG$
the process $(Y_{t+\tau_k}/Y_{\tau_k})$
is a geometric Brownian motion.
\item
Under $\QQ$ conditionally on $\cG$
the random variable $R_{k+1}$
is uniformly distributed on $[0,1]$.
\item
Under $\QQ$ conditionally on $\cG$
the process $(Y_{t+\tau_k}/Y_{\tau_k})$
and the random variable $R_{k+1}$ are independent
(this follows from~\eqref{eq:p11}).
\end{enumerate}
Therefore, by Lemma~\ref{lem:p1}
applied under $\QQ$ conditionally on $\cG$,
for any $x\in\bbR_+$,
$\QQ$-a.s. it holds
$$
\QQ\left(\left.\frac{Y_{\tau_{k+1}}}{Y_{\tau_k}}\leq x
\right|\cG\right)
=F_{k+1}(x|Y_{\tau_0},\ldots,Y_{\tau_k}).
$$
By~\eqref{eq:p6},
$\PP$-a.s. on $\{\tau_k<\infty\}$ it holds
\begin{equation}
\label{eq:p7}
\PP\left(\left.\frac{Y_{\tau_{k+1}}}{Y_{\tau_k}}\leq x
\right|\cG\right)
=F_{k+1}(x|Y_{\tau_0},\ldots,Y_{\tau_k}).
\end{equation}
But $\PP$-a.s. on
$\{\tau_k=\infty\}(\equiv\{Y_{\tau_k}=0\})$
we have $Y_{\tau_{k+1}}=0$,
i.e. the left-hand side
of~\eqref{eq:p7} is then $1_{\{x\geq1\}}$,
which coincides with the right-hand side
of~\eqref{eq:p7} on this event.
Thus, \eqref{eq:p7} holds
$\PP$-a.s. on $\Omega$
(not only on~$\{\tau_k<\infty\}$).
Since $x\in\bbR_+$ is arbitrary,
this implies that
$F_{k+1}(\cdot|Y_{\tau_0},\ldots,Y_{\tau_k})$
is a version of the
regular conditional distribution function
(under~$\PP$)
of $Y_{\tau_{k+1}}/Y_{\tau_k}$
given $Y_{\tau_0},\ldots,Y_{\tau_k}$.
Together with~\eqref{eq:p3}
and the definition of $F_{k+1}$
this implies~\eqref{eq:p5}.
The induction step is proved.

Thus, we can construct
a nondecreasing family $(\tau_n)$
of $(\cF_t)$-stopping times
such that the discrete-time processes
$(X_n)$ and $(Y_{\tau_n})$
have the same finite-dimensional distributions.
This completes the proof of the lemma.
\end{proof}

%=======
\subsection{Continuous-Time Limit.}
\label{par:pct}
Let us proceed with the proof of Theorem~\ref{th:i1}.
We are now given a continuous-time
nonnegative supermartingale $(X_s)$
with $\EE X_0\leq1$.
For each $n\in\bbN$,
let us consider the piecewise
constant nonnegative supermartingale
$(X^n_s)
=\big(X_{2^{-n}\lfloor 2^n s\rfloor}\big)$.
By Lemma~\ref{lem:p2},
there exists a geometric Brownian motion
$(Y^n_t)$ and a (piecewise constant)
time change $(T^n_s)$
on some filtered probability space
$(\Omega^n,\cF^n,(\cF^n_t),\PP^n)$
such that the processes
$(X^n_s)$ and $(Y^n_{T^n_s})$
have the same law.
Without loss of generality we assume that
$\lim_{t\to\infty}Y^n_t(\omega_n)=0$
for all $\omega_n\in\Omega^n$
and set $Y^n_\infty(\omega_n):=0$
for all $\omega_n\in\Omega^n$.

Let $C([0,\infty])$ be the space of all continuous functions $z\colon[0,\infty]\to\bbR$ with the sup-norm, and let $\cA$ be the set of all non-decreasing right-continuous functions $a\colon[0,\infty)\to [0,\infty]$. Define a metric $\rho$ on $\cA$ by
$\rho(a_1,a_2) = d(\hat a_1, \hat a_2)$, where $\hat a_i = \frac{a_i}{1+a_i}$ and
\[
d(b_1,b_2)=\sum_{k=1}^\infty 2^{-k}\int_0^k | b_1(t) - b_2(t)|\dd t.
\]
It is easy to check that the convergence in the metric $d$ in the space $\hat\cA$ of all non-decreasing right-continuous functions $b\colon[0,\infty)\to [0,1]$ is equivalent to the pointwise convergence for every point at which the limiting function is continuous. By Helly's theorem, $(\hat \cA,d)$ is a compact. Hence, $(\cA,\rho)$ is a compact.

Put $\Omega = C([0,\infty])\times\cA$, $\cF = \cB(C([0,\infty]))\otimes\cB(\cA)$. The space $\Omega$ with the product topology is a complete separable metric space. We define the measurable mapping
$f_n\colon\Omega^n\to\Omega$ by
\[
f_n(\omega_n) = (Y^n(\omega_n),T^n(\omega_n)).
\]
Let $\QQ^n$ be the image of $\PP^n$ under $f_n$. First, we show that the sequence $\QQ_n$ of probability measures on $(\Omega,\cF)$ is tight. It is sufficient to check that the projections of $\QQ^n$ on $C([0,\infty])$ and on $\cA$ are tight. The projection of $\QQ^n$ on
$C([0,\infty])$ is the law of a geometric Brownian motion and does not depend on $n$, which implies the tightness of the projections on $C([0,\infty])$. The tightness of projections on $\cA$ follows from the compactness of $\cA$. Thus, the sequence $\QQ^n$ is tight. Now we define $\PP$ as an accumulation point of this sequence. It is evident that the process $(Z_t)_{t\in\bbR_+}$
on~$\Omega$ defined by $Z_t(z,a)=z(t)$ is a (standard)
geometric Brownian motion under~$\PP$:
namely, the process $(W_t)_{t\in\bbR_+}$
with $W_t=\log Z_t+t/2$,
which is well-defined under~$\PP$,
is a Brownian motion under~$\PP$.

Define the process $(T_s)_{s\in\bbR_+}$
on $\Omega$ by $T_s(z,a)=a(s)$
and consider the minimal right-continuous
filtration $(\cF_t)_{t\in\bbR_+}$ on $\Omega$
with respect to which
$(Z_t)$ is adapted and
$(T_s)$ is a time change, i.e.
\[
\cF_t = \bigcap_{\eps>0} \sigma
\big(Z_u,\{T_s\leq v\}\colon
u,v\in[0,t+\eps], s\in\bbR_+\big).
\]
The remaining steps of the proof are to show that:
\begin{enumerate}
\item
The process $(Z_{T_s})$ has the same law
under $\PP$ as $(X_s)$.
\item
The process $(W_t)$ is
an $(\cF_t,\PP)$-Brownian motion,
i.e. $W_t-W_s$ is independent of $\cF_s$
under~$\PP$
for any $s<t$, $s,t\in\bbR_+$.
\end{enumerate}
These two steps are proved similarly to the corresponding steps in the proof of Theorem~2
in~\cite{Monroe:78}
with obvious changes.

One can also give an alternative proof using a version of Theorem (3.2) in \cite{BaxterChacon:77}. The idea is to introduce a kind of stable topology on $\Omega$ such that (2) remains true after passing to the limit; however, then the compactness is a nontrivial issue.

%=======
\section{SEP, Minimality and Embedding of Processes}
\label{sec:app}
\subsection{Classical SEP
and Minimal Stopping Times}
We start with a few remarks
on the evolution of the formulation
of the embedding problem for a Brownian motion.

\begin{problem}[Embedding in a Brownian motion,
naive formulation]
\label{pr:app1}
\mbox{}\\
Given: a real-valued random variable $\xi$.\\
To find: a filtered probability space
$(\Omega,\cF,(\cF_t)_{t\geq0},\PP)$,
an $(\cF_t,\PP)$-Brownian motion $B=(B_t)$
and a finite
$(\cF_t)$-stopping time $\tau$
such that $B_\tau\sim\xi$.
\end{problem}

Problem~\ref{pr:app1}
admits the following trivial solution.
Let $f\colon\bbR\to\bbR$
be a function such that
$f(B_1)\sim\xi$.
Then, with
\begin{equation}
\label{eq:app1}
\tau:=\inf\{t\geq1:B_t=f(B_1)\},
\end{equation}
due to recurrence of a Brownian motion,
we have
$B_\tau=f(B_1)\sim\xi$.
This solution is attributed to Doob
(see the discussion in Section~2.3
in~\cite{Obloj:04}
or Section~3.2 in~\cite{Hobson:11})
and is intended to show that without
additional requirements
the problem is trivial.

Therefore, the original formulation
of the SEP contains some restrictions:

\begin{problem}[SEP,
Skorokhod~\cite{Skorokhod:61}
and~\cite{Skorokhod:65}]
\label{pr:app2}
\mbox{}\\
Given: a real-valued random variable $\xi$
\emph{with $\EE\xi=0$ and $\EE\xi^2<\infty$.}\\
To find: a filtered probability space
$(\Omega,\cF,(\cF_t)_{t\geq0},\PP)$,
an $(\cF_t,\PP)$-Brownian motion $B=(B_t)$
and an $(\cF_t)$-stopping time $\tau$
\emph{with $\EE\tau<\infty$}
such that $B_\tau\sim\xi$.
\end{problem}

Note that the stopping time $\tau$ of~\eqref{eq:app1}
is excluded because, for this stopping time,
$\EE\tau=\infty$
(unless $f$ is the identity function,
which is only possible when $\xi\sim N(0,1)$).
Let us further note that
$\EE\tau<\infty$
implies $\EE B_\tau=0$ and $\EE B_\tau^2=\EE\tau$,
hence we need to assume
$\EE\xi=0$ and $\EE\xi^2<\infty$
in the formulation
when we have the requirement
$\EE\tau<\infty$.
However,
these assumptions
($\EE\xi=0$ and $\EE\xi^2<\infty$)
constitute
the drawback of the formulation
in Problem~\ref{pr:app2}.
For example, the stopping times
in the original Skorokhod's construction
(see~\cite{Skorokhod:61} and~\cite{Skorokhod:65})
require from $\xi$ only to have a finite mean,
but, as we have just seen,
unless we assume a finite variance,
it is no longer clear
how to select ``good'' stopping times.

A very natural way to select ``good''
stopping times is to require them
to be minimal
(instead of requiring $\EE\tau<\infty$)
in the following sense.
A finite stopping time $\tau$
is said to be \emph{minimal}
if, for a stopping time~$\sigma$,
$\sigma\leq\tau$ and $B_\sigma\sim B_\tau$
imply $\sigma=\tau$~a.s.
In the context of the SEP,
this was suggested in~\cite{Monroe:72}
(such a concept of minimality
is attributed by Monroe~\cite{Monroe:72} to Doob)
and taken on in many subsequent works on the SEP.
In particular,
for centred target distributions,
minimality is characterised
in~\cite{Monroe:72}
as follows.

\begin{theorem}[Monroe~\cite{Monroe:72}]
\label{th:app1}
Let $\tau$ be a finite
stopping time such that
$\EE B_\tau=0$.
Then $\tau$ is minimal if and only if
the process $(B_{t\wedge\tau})_{t\geq0}$
is uniformly integrable.
\end{theorem}

This characterisation proved to be very useful.
We will also need it below.
Let us further remark that minimality
for non-centred target distributions
was characterised in~\cite{CoxHobson:06},
in particular, Theorem~\ref{th:app1}
was generalised for $\EE|B_\tau|<\infty$.
See also
Section~8 in~\cite{Obloj:04},
Sections~3.4 and~4.2
in~\cite{Hobson:11}
as well as
Section~2.2 in~\cite{Obloj:12}
for a further discussion of minimality.

Summarising,
\cite{Monroe:72} and~\cite{CoxHobson:06}
inspire the following formulation of the SEP:

\begin{problem}[SEP,
Monroe~\cite{Monroe:72},
Cox and Hobson~\cite{CoxHobson:06}]
\label{pr:app3}
\mbox{}\\
Given: a real-valued random variable $\xi$.\\
To find: a filtered probability space
$(\Omega,\cF,(\cF_t)_{t\geq0},\PP)$,
an $(\cF_t,\PP)$-Brownian motion $B=(B_t)$
and a \emph{minimal}
$(\cF_t)$-stopping time $\tau$
such that $B_\tau\sim\xi$.
\end{problem}

Let us note that Problem~\ref{pr:app3}
is more general than Problem~\ref{pr:app2}
in the sense that each solution
of Problem~\ref{pr:app2}
is a solution of Problem~\ref{pr:app3}
(if $\sigma\leq\tau$ are solutions
of Problem~\ref{pr:app2},
then $\EE\tau=\EE\sigma=\EE\xi^2<\infty$,
i.e. $\sigma=\tau$~a.s.,
hence $\tau$ is minimal),
but we do not assume
$\EE\xi=0$ and $\EE\xi^2<\infty$
any longer.

%=======
\subsection{Embedding of Processes}%
\mbox{}
Here we explain why Theorem~\ref{th:i1}
is not a consequence of~\cite{Monroe:72}
and~\cite{Monroe:78}.
In fact, what can be inferred directly
from Monroe's results is only the following
(weaker) statement.

\begin{proposition}
\label{prop:app1}
Let $X=(X_s)_{s\geq0}$
be a nonnegative martingale with $\EE X_0=1$.
Then there exists a filtered probability space
$(\Omega,\cF,(\cF_t)_{t\geq0},\PP)$,
an $(\cF_t,\PP)$-Brownian motion $W=(W_t)$
and a $[0,\infty]$-valued
$(\cF_t)$-time change $(T_s)$
such that the processes
$(X_s)_{s\geq0}$ and $(Z_{T_s})_{s\geq0}$
have the same law,
where $Z_t=e^{W_t-t/2}$, $t\geq0$.
\end{proposition}

Let us remark that nonnegative supermartingales
$(X_s)$ with $\EE X_0\leq1$
(see Theorem~\ref{th:i1}~(i))
is an important class of processes,
which naturally appears
in different branches of stochastics
such as financial mathematics
or sequential analysis.
As for financial mathematics,
so-called supermartingale deflators appear naturally as an extension of the class of the density processes of equivalent martingale measures. In particular, existence of a strictly positive supermartingale deflator is a weaker assumption than existence of equivalent (local) martingale measure and is equivalent to some form of absence of arbitrage, see~\cite{KaratzasKardaras:07}. Even if an equivalent local martingale measure exists, it is necessary to use supermartingale deflators in the utility maximization problem, see~\cite{KramkovSchachermayer:99}.
As for sequential analysis,
let us, for instance, note
that given two probability measures
$\PP$ and $\QQ$
on a filtered space $(\Omega,\cF,(\cF_t))$
the generalised density process
$(\frac{d\QQ_t}{d\PP_t})$
is, in general, only a supermartingale
under $\PP$
(a martingale only when $\QQ$
is locally absolutely continuous
with respect to~$\PP$).
Thus, it is really better to have
Theorem~\ref{th:i1}
than just Proposition~\ref{prop:app1}.

Turning to the discussion
of the relations with Monroe's results,
let us first recall that both
in~\cite{Monroe:72}
and in~\cite{Monroe:78}
the question is treated
of whether a process is equivalent
to a time-changed Brownian motion.
The difference is that, in~\cite{Monroe:72},
only finite time changes
consisting of minimal stopping times,
while in~\cite{Monroe:78},
all finite time changes are considered.
Therefore, the results are very different:

\begin{theorem}[Monroe~\cite{Monroe:72}]
\label{th:app2}
Let $M=(M_s)_{s\geq0}$ be a martingale.
Then there is a filtered probability space
$(\Omega,\cF,(\cF_t)_{t\geq0},\PP)$,
an $(\cF_t,\PP)$-Brownian motion
$W=(W_t)$
and a finite $(\cF_t)$-time change $(T_s)$
such that all stopping times $T_s$
are minimal and the processes
$(M_s)$ and $(W_{T_s})$
have the same law.
\end{theorem}

\begin{theorem}[Monroe~\cite{Monroe:78}]
\label{th:app3}
A c\`adl\`ag process $X=(X_s)_{s\geq0}$
is a semimartingale if and only if
there is a filtered probability space
$(\Omega,\cF,(\cF_t)_{t\geq0},\PP)$,
an $(\cF_t,\PP)$-Brownian motion
$W=(W_t)$
and a finite $(\cF_t)$-time change $(T_s)$
such that the processes
$(X_s)$ and $(W_{T_s})$
have the same law.
\end{theorem}

We now know what kind of processes
can be viewed as time changes of a Brownian motion
(Theorems~\ref{th:app2} and~\ref{th:app3}),
while we are interested in understanding of
what kind of processes can be viewed as
time changes of the geometric Brownian motion
$Z=(Z_t)$ of Theorem~\ref{th:i1}~(i).
The idea is first to change time in $Z$
in order to get a Brownian motion starting from one
(to which we then want to apply Monroe's results),
but it is only possible to obtain a Brownian motion
absorbed at zero.
More precisely,
with $A_t:=[Z,Z]_t=\int_0^t Z_r^2\,dr$,
we define the time change
$\tau_u:=\inf\{r\geq0:A_r>u\}$
and set
\begin{equation}
\label{eq:app2}
B^0_u:=Z_{\tau_u},\quad u\geq0.
\end{equation}
Note that $A_\infty<\infty$~a.s.
and that $(\tau_u)$ is strictly increasing
on $[0,A_\infty)$ and is equal
to $+\infty$ on $[A_\infty,\infty)$.
We have:
$B^0=(B^0_u)_{u\geq0}$
is a Brownian motion absorbed at zero
with $B^0_0=1$
(see \cite[Ch.~V, \S~1]{RevuzYor:99}).
Now the idea is:
if we can embed a process $X$
in $B^0$ in the sense that
$(X_s)$ and $(B^0_{T_s})$
have the same law for some time change $(T_s)$,
then we can embed $X$ in $Z$
via the time change $(\tau_{T_s})$
(see~\eqref{eq:app2}).
At this point Theorem~\ref{th:app2}
turns out to be very useful
and gives us Proposition~\ref{prop:app1}.
Namely, let $X$ be a nonnegative martingale
with $\EE X_0=1$.
Applying Theorem~\ref{th:app2}
to the martingale $X-1$
we get that $(X_s)$ has the same law as
$(Y_{T_s})$ for some Brownian motion $Y$
starting from one and a time change $(T_s)$
such that all stopping times $T_s$ are minimal.
By Theorem~\ref{th:app1},
for each $s\geq0$,
the process $(Y_{u\wedge T_s})_{u\geq0}$
is uniformly integrable.
(The condition $\EE B_\tau=0$ in Theorem~\ref{th:app1}
takes here the form $\EE Y_{T_s}=1$
because $Y$ starts from one.
This is fulfilled because $X$ is a martingale.)
Since $X_s\geq0$~a.s., we have
$Y_{T_s}\geq0$~a.s., hence
the uniformly integrable martingale
$(Y_{u\wedge T_s})_{u\geq0}$
is nonnegative, which implies
$$
T_s\leq H^Y_0:=\inf\{u\geq0:Y_u=0\}\quad\text{a.s.}
$$
Therefore, $Y_{T_s}=Y^0_{T_s}$~a.s.
for all $s\geq0$, where
$Y^0:=(Y_{u\wedge H^Y_0})_{u\geq0}$
is the Brownian motion $Y$ stopped
at the time it hits zero.
Thus, $X$ can be embedded
in the absorbed Brownian motion
$Y^0$, i.e. this idea works.
There remain some technical details,
but it is already clear that, indeed,
Proposition~\ref{prop:app1}
can be inferred from Monroe's results,
namely, from Theorems~\ref{th:app2}
and~\ref{th:app1}.

On the contrary,
such an argumentation
does not work any longer
if we try to obtain Theorem~\ref{th:i1}~(i)
from Theorem~\ref{th:app3}.
Indeed, let $X$ be a nonnegative
supermartingale with $\EE X_0\leq1$.
Applying Theorem~\ref{th:app3}
to the semimartingale $X-1$
we get that $(X_s)$ has the same law as
$(Y_{T_s})$ for some Brownian motion $Y$
starting from one and a time change $(T_s)$.
But now there is no reason
for stopping times $T_s$ to be minimal.
We need to justify that $T_s\leq H^Y_0$
with $ H^Y_0$ defined as above,
but it was minimality of $T_s$
together with the property $\EE Y_{T_s}=1$
that previously gave us
the desired inequality $T_s\leq H^Y_0$.
In the situation of Theorem~\ref{th:app3},
it can happen that the desired inequality
fails even when we start with
a nonnegative supermartingale
$X$ with $\EE X_0\leq1$
(one can easily construct such examples
due to recurrence of the Brownian motion).
Thus, what we need is to justify
that whenever $X$ is a nonnegative
supermartingale with $\EE X_0\leq1$,
then it is possible not only to find
some time change $(T_s)$
as stated in Theorem~\ref{th:app3},
but rather a time change
with the additional property
$T_s\leq H^Y_0$.
However, the latter statement
is beyond the scope
of Monroe's theorems.

Moreover, the following statement,
which complements Theorem~\ref{th:app2},
is a direct consequence of our Theorem~\ref{th:i1}.

\begin{theorem}
\label{th:app4}
Let $X=(X_s)_{s\geq0}$
be a supermartingale bounded from below with $\EE X_0\leq0$.
Then there is a filtered probability space
$(\Omega,\cF,(\cF_t)_{t\geq0},\PP)$,
an $(\cF_t,\PP)$-Brownian motion
$W=(W_t)$
and a finite $(\cF_t)$-time change $(T_s)$
such that all stopping times $T_s$
are minimal and the processes
$(X_s)$ and $(W_{T_s})$
have the same law.
\end{theorem}

\begin{proof}
Let $c>0$ and $X_s\geq-c$ for all $s\geq0$.
By Theorem~\ref{th:i1}, the process $(c^{-1}X_s+1)_{s\geq0}$ is equivalent to a time-changed
geometric Brownian motion $(Z_{\sigma_s})_{s\geq0}$ given on a filtered probability space
$(\tilde\Omega,\tilde\cF,(\tilde\cF_t)_{t\geq0},\tilde\PP)$. Put $A_t:=[cZ,cZ]_t=c^2\int_0^t Z_r^2\,dr$ and $\tau_u:=\inf\{r\geq0:A_r>u\}$. As above, $A_\infty<\infty$~a.s. and $(\tau_u)$ is strictly increasing
on $[0,A_\infty)$ and is equal
to $+\infty$ on $[A_\infty,\infty)$. By the Dambis--Dubins--Schwarz theorem, see \cite[Ch.~V, Theorem 1.7]{RevuzYor:99}, there is a standard Brownian motion $W=(W_t)_{t\geq0}$ on an enlargement $(\Omega,\cF,(\cF_t)_{t\geq0},\PP)$ of $(\tilde\Omega,\tilde\cF,(\tilde\cF_{\tau_t})_{t\geq0},\tilde\PP)$ such that, for all $t\geq 0$,
\[
c+W_{t\wedge A_\infty} = cZ_{\tau_t}
 \text{ and, therefore, } c+W_{A_t} = cZ_t.
\]
Then $T_s:=A_{\sigma_s}$ is a time change with respect to $(\tilde\cF_{\tau_s})$ and hence to $(\cF_s)$,
\begin{equation}
\label{eq:app3}
T_s\leq T_\infty\leq A_\infty=\inf\{t\geq0:W_t=-c\}
\;\;(\equiv H^W_{-c})
\end{equation}
(in particular, $T_s$ are finite),
and $(X_s)_{s\geq0}$ is equivalent
to $(W_{T_s})_{s\geq0}$.

Finally, the fact that all $T_s$, $s\geq 0$,
are minimal follows
via~\eqref{eq:app3}
from Theorem~\ref{th:mst1} below
or from Theorem~5 in~\cite{CoxHobson:06}.
\end{proof}

The above discussion shows that Theorem~\ref{th:i1} can be deduced from Theorem~\ref{th:app4} as well
(in place of Theorem~\ref{th:app1}
use Theorem~5 in~\cite{CoxHobson:06}).

%=======
\section{Minimal Stopping Times
for Other Processes}
\label{sec:mst}
Above we discussed only minimal
stopping times for a Brownian motion,
but one can similarly consider
minimality of a stopping time for any process
(cf.~Section~3.4 in~\cite{Hobson:11}).

In this section, we consider a state space
$(E,\cE)$,
where $E$ is $[l,r]$ with $-\infty\leq l<r\leq\infty$
or $\bbR^d\cup\{\infty\}$
and $\cE$ is the Borel $\sigma$-field
on~$E$.\footnote{As for the topology
on $\bbR^d\cup\{\infty\}$ that we consider,
the neighbourhood system for $\infty$
in $\bbR^d\cup\{\infty\}$
is the family of the complements
of the compact sets in $\bbR^d$.}
It may be convenient
that the state space contains infinite points
in order to treat stopping times
that can take infinite value.

\begin{definition}
\label{def:mst1}
Let $X=(X_t)_{t\geq0}$
be an $E$-valued adapted c\`adl\`ag process
on a filtered probability space
$(\Omega,\cF,(\cF_t)_{t\geq0},\PP)$.
An $(\cF_t)$-stopping time $\tau$
is said to be \emph{minimal for $X$}
if, for an $(\cF_t)$-stopping time~$\sigma$,
$\sigma\leq\tau$ and $X_\sigma\sim X_\tau$
imply $\sigma=\tau$~a.s.
The limit $X_\infty:=\lim_{t\to\infty}X_t$
will exist a.s. on the set $\{\tau=\infty\}$
whenever minimality of a stopping time
$\tau$ with $\PP(\tau=\infty)>0$ is checked
(so that $X_\tau$ and $X_\sigma$ are well-defined).
\end{definition}

Let us remark that, e.g.,
for a Brownian motion with a non-zero drift,
the natural state space is $\ol\bbR:=[-\infty,\infty]$.
This allows to check every stopping time
for minimality and not a priori to exclude
stopping times that can take infinite value.
In this connection, let us also notice that,
for a Brownian motion with a non-zero drift,
every stopping time is minimal,
which follows from the next theorem
(this is different from the case of a Brownian motion,
cf.~Section~\ref{sec:app}).

\begin{theorem}
\label{th:mst1}
Let $\tau$ be an $(\cF_t)$-stopping time
and $g\colon E\to\ol\bbR$
a measurable function
such that the following holds:
\begin{enumerate}[(a)]
\item\label{it:msta}
the stopped process
$g(X)^\tau=(g(X_{t\wedge\tau}))_{t\geq0}$
is a \emph{closed} supermartingale
(i.e. $g(X)^\tau$ is a supermartingale
bounded from below by a uniformly
integrable martingale),
\item\label{it:mstb}
a.s. $g(X)$ has no intervals of constancy
on the stochastic interval $[0,\tau)$,
\item\label{it:mstc}
a.s. on $\{\tau=\infty\}$ there exists
$X_\infty:=\lim_{t\to\infty}X_t$.
\end{enumerate}
Then $\tau$ is minimal for~$X$.
\end{theorem}

\begin{remark}
\label{rem:mst1}
Let $E=[l,r]$ and $g$ be strictly monotone.
Then $\tau$ is minimal
whenever only \eqref{it:msta}
and~\eqref{it:mstb} hold
(in other words, condition~\eqref{it:mstc}
can be dropped in this case).
Indeed, if $(X_t)_{t\geq0}$
had distinct limit points
as $t\to\infty$ on $\{\tau=\infty\}$,
then $(g(X_t))_{t\geq0}$
would have distinct limit points as well.
But the latter is not the case because,
by~\eqref{it:msta},
the limit $\lim_{t\to\infty}g(X_t)$
exists a.s. on $\{\tau=\infty\}$
($g(X)^\tau$ converges a.s.
as a closed supermartingale).
\end{remark}

\begin{proof}[Proof of Theorem~\ref{th:mst1}]
Without loss of generality we assume below
that $g$ is the identity function
(otherwise pursue the reasoning below
with $g(X)$ in place of~$X$).

The proof is a combination of two following arguments.

(1) For a closed supermartingale~$Y$,
Doob's optional sampling theorem
works with arbitrary stopping times,
i.e., for any stopping times $\rho\leq\eta$,
we have $Y_\rho,Y_\eta\in L^1$
and ${\EE(Y_\eta|\cF_\rho)\leq Y_\rho}$~a.s.

(2) If $\xi_1\leq\xi_2$ are random variables
in $L^1$ with $\EE\xi_1=\EE\xi_2$,
then $\xi_1=\xi_2$~a.s.

\smallskip
Suppose that $\sigma$
is a stopping time with
$\sigma\leq\tau$ and $X_\sigma\sim X_\tau$.
Then, by arguments (1) and~(2),
$\EE(X_\tau|\cF_\sigma)=X_\sigma$~a.s.
Take a strictly convex function $h$
of linear growth,
e.g. $h(x)=\sqrt{1+x^2}$.
By Jensen's inequality and argument~(2),
$\EE(h(X_\tau)|\cF_\sigma)=h(X_\sigma)$~a.s.,
i.e. we have the equality in Jensen's
inequality with a strictly convex function.
Then $X_\tau=\EE(X_\tau|\cF_\sigma)$~a.s.,
i.e. $X_\tau=X_\sigma$~a.s.

Let $\rho$ be any stopping time with
$\sigma\leq\rho\leq\tau$. Then
$$
X_\rho=\EE(X_\tau|\cF_\rho)
=\EE(X_\sigma|\cF_\rho)
=X_\sigma\quad\text{a.s.},
$$
where the first equality is due to
arguments (1) and~(2)
(use
${\EE(X_\tau|\cF_\sigma)
\leq\EE(X_\rho|\cF_\sigma)
\leq X_\sigma}$~a.s.).
Since $X$ has no intervals of constancy
on $[0,\tau)$,
we get $\sigma=\tau$~a.s.
\end{proof}

\begin{remark}
\label{rem:mst2}
Theorem~\ref{th:mst1}
can be slightly generalised as follows.
The word ``supermartingale''
in~\eqref{it:msta}
should be understood
as a c\`adl\`ag process
that is a supermartingale
in the sense of Definition~(1.1)
in~\cite[Ch.~II]{RevuzYor:99}
and the following assumption
should be added:
\begin{enumerate}[(a)]
\setcounter{enumi}{3}
\item\label{it:mstd}
$g(X_\tau)\in L^1$.
\end{enumerate}
This slightly more general definition
of a supermartingale
(applied to a process~$Y$)
differs from the usual one
in that only $Y_t^-\in L^1$,
$t\geq0$, is required,
while $Y_t$ can be non-integrable
(and can even take value $\infty$
with a positive probability).
The resulting statement is slightly
stronger than Theorem~\ref{th:mst1}
(in Theorem~\ref{th:mst1},
\eqref{it:mstd} is satisfied automatically,
see argument~(1) in the proof),
but the formulation of Theorem~\ref{th:mst1}
is more transparent in the present form.
The same proof applies with the only difference:
in argument~(1) we only have
$Y_\rho^-,Y_\eta^-\in L^1$,
but, due to~\eqref{it:mstd},
we always can use argument~(2)
when we need it.
\end{remark}

In the examples below we will see
that Theorem~\ref{th:mst1}
applies in many specific situations.
We will also need the following lemma
(its proof is straightforward).

\begin{lemma}
\label{lem:mst1}
Let $Y=(Y_t)_{t\geq0}$ be a supermartingale.
Then
$$
Y
\text{ is a closed supermartingale }
\Longleftrightarrow
\text{ the family }
(Y_t^-)_{t\geq0}
\text{ is uniformly integrable.}
$$
\end{lemma}

In Examples~\ref{ex:mst1} and~\ref{ex:mst2} below,
$X$ will be a one-dimensional diffusion.
To this end, we introduce some notations.
Let $J=(l,r)$, $-\infty\leq l<r\leq\infty$,
and $E=[l,r]$.
We consider a time-homogeneous
diffusion $X$ in $J$
being a solution of the SDE
\begin{equation}
\label{eq:mst10}
dX_t=\mu(X_t)\,dt+\sigma(X_t)\,dW_t,
\quad X_0=x_0\in J,
\end{equation}
on some filtered probability space
$(\Omega,\cF,(\cF_t)_{t\geq0},\PP)$,
where $W$ is an $(\cF_t)$-Brownian motion.
We assume that the coefficients
$\mu$ and $\sigma$
are Borel-measurable functions that satisfy
\begin{gather}
\label{eq:mst20}
\sigma(x)\neq0\;\;\forall x\in J,\\
\label{eq:mst30}
\frac1{\sigma^2},\frac\mu{\sigma^2}\in L^1_\loc(J),
\end{gather}
where $L^1_\loc(J)$ denotes the set
of locally integrable on $J$ functions.
Under \eqref{eq:mst20} and~\eqref{eq:mst30}
SDE~\eqref{eq:mst10}
has a weak solution, unique in law,
which possibly exits~$J$
(see~\cite[Sec.~5.5]{KaratzasShreve:91}).
The exit time is denoted by $\zeta$.
That is to say,
a.s. on $\{\zeta=\infty\}$
the trajectories of $X$ do not exit~$J$,
while a.s. on $\{\zeta<\infty\}$ we have:
either $\lim_{t\nearrow\zeta}X_t=r$
or $\lim_{t\nearrow\zeta}X_t=l$.
We specify the behaviour of $X$
after $\zeta$ on $\{\zeta<\infty\}$
by making $l$ and $r$
be absorbing boundaries.
Thus, we get an $E$-valued
process $X=(X_t)_{t\geq0}$.
For some $c\in J$, we set
$$
s(x)=\int_c^x
\exp\left\{
-\int_c^y\frac{2\mu}{\sigma^2}(z)\,dz
\right\}\,dy,
\quad x\in E
\;\;(\equiv[l,r]),
$$
which is a scale function of~$X$
(any scale function of $X$
is an affine transformation
of $s$ with a strictly positive slope).
Let us note that, on~$J$,
$s$ is a strictly increasing $C^1$-function
with a strictly positive absolutely
continuous derivative,
while~$s(r)$ (resp.~$s(l)$)
may take value~$\infty$ (resp.~$-\infty$).
Finally, we recall that $s(X)$
is an $(\cF_t)$-local martingale
(the boundary, at which the scale function
is infinite, is not attained).

\begin{example}[One-dimensional diffusion,
transient case]
\label{ex:mst1}
Assume that $s(r)\wedge|s(l)|<\infty$.
Then $s(X)$ is a local martingale
bounded from below (if~$s(l)>-\infty$)
or from above (if~$s(r)<\infty$),
hence a closed super- or submartingale.
Theorem~\ref{th:mst1}
with $g$ being $s$ or $-s$
implies that,
under $s(r)\wedge|s(l)|<\infty$,
$$
\text{every }
(\cF_t)
\text{-stopping time }
\tau
\text{ such that }
\tau\leq\zeta
\text{ a.s. is minimal for }
X.
$$
(Notice that,
by It\^o's formula applied to $s(X)$,
assumption~\eqref{it:mstb}
in Theorem~\ref{th:mst1}
follows from~\eqref{eq:mst20},
while \eqref{it:mstc} need not be checked
due to Remark~\ref{rem:mst1}.)
\end{example}

\begin{remark}
\label{rem:mst3}
Let $a\neq0$ and $B$
be an $(\cF_t)$-Brownian motion
on some filtered probability space.
Set $Y_t=B_t+at$, $t\geq0$.
It follows from the previous example
that every $(\cF_t)$-stopping time is minimal for~$Y$
(and for the geometric Brownian motion $e^Y$).
In particular, contrary to the Brownian case,
when considering the SEP for the geometric
Brownian motion $(e^{B_t-t/2})$,
as we did in Lemma~\ref{lem:p1},
there is no difference between
setting the problem like Problem~\ref{pr:app1}
or like Problem~\ref{pr:app3}
in Section~\ref{sec:app}.
\end{remark}

\begin{example}[One-dimensional diffusion,
recurrent case]
\label{ex:mst2}
Assume that $s(r)=-s(l)=\infty$.
Then $\zeta=\infty$~a.s.
and $\limsup_{t\to\infty}X_t=r$~a.s.,
$\liminf_{t\to\infty}X_t=l$~a.s.
In particular, in this example
minimality is well-defined
only for finite $(\cF_t)$-stopping times.
We first assume that the local martingale
$Y:=s(X)$ is, in fact, a true martingale.
Let us note that $Y$ satisfies the SDE
\begin{equation}
\label{eq:mst35}
dY_t=\ka(Y_t)\,dW_t,
\quad Y_0=y_0:=s(x_0),
\end{equation}
where $\ka:=(s'\sigma)\circ s^{-1}$
is a Borel-measurable function satisying
\begin{equation}
\label{eq:mst40}
\ka(x)\neq0\;\;\forall x\in\bbR,
\quad\ka^{-2}\in L^1_\loc(\bbR).
\end{equation}
It follows from~\cite{Kotani:06}
that $Y$ is a martingale if and only if
\begin{equation}
\label{eq:mst50}
\int_c^\infty\frac x{\ka^2(x)}\,dx=\infty
\text{ and }
\int_{-\infty}^c\frac{|x|}{\ka^2(x)}\,dx=\infty
\end{equation}
with some $c\in\bbR$
(condition~\eqref{eq:mst50} does not depend
on $c$ due to~\eqref{eq:mst40}).
Now Theorem~\ref{th:mst1}
with $g$ being $s$ or $-s$
and Lemma~\ref{lem:mst1}
imply that,
under $s(r)=-s(l)=\infty$ and~\eqref{eq:mst50},
any $(\cF_t)$-stopping time $\tau$
satisfying
\begin{equation}
\label{eq:mst60}
\text{either }
(s(X_{t\wedge\tau})^-)_{t\geq0}
\text{ or }
(s(X_{t\wedge\tau})^+)_{t\geq0}
\text{ is uniformly integrable}
\end{equation}
is finite and minimal for~$X$.
(We also get the finiteness
of $\tau$ from~\eqref{eq:mst60}
because the closed super-
or submartingale $s(X)^\tau$
converges~a.s.,
see Remark~\ref{rem:mst1}.)
Finally, if we no longer assume~\eqref{eq:mst50},
then any $(\cF_t)$-stopping time $\tau$
satisfying
\begin{equation}
\label{eq:mst65}
\text{either }
\EE\sup_{t\geq0}
s(X_{t\wedge\tau})^-<\infty
\text{ or }
\EE\sup_{t\geq0}
s(X_{t\wedge\tau})^+<\infty
\end{equation}
is finite and minimal for~$X$.
(Under~\eqref{eq:mst65},
$s(X)^\tau$ is a closed super-
or submartingale
as a local martingale bounded
from below or from above
by an integrable random variable.)
\end{example}

\begin{remark}
\label{rem:mst4}
Let $B$
be an $(\cF_t)$-Brownian motion
on some filtered probability space.
It follows from the previous example that
any $(\cF_t)$-stopping time $\tau$
satisfying
\begin{equation}
\label{eq:mst70}
\text{either }
(B_{t\wedge\tau}^-)_{t\geq0}
\text{ or }
(B_{t\wedge\tau}^+)_{t\geq0}
\text{ is uniformly integrable}
\end{equation}
is finite and minimal for~$B$.
We now recall that,
by Theorem~3 in~\cite{CoxHobson:06},
under the assumption $\EE|B_\tau|<\infty$,
\eqref{eq:mst70} is, in fact, equivalent
to the minimality of~$\tau$.
(Let us also notice that
\eqref{eq:mst70} implies that
$B^\tau$ is a closed super- or submartingale,
hence $\EE|B_\tau|<\infty$.)
Thus, for a Brownian motion,
sufficient condition~\eqref{eq:mst70}
that we get from
Theorem~\ref{th:mst1}
turns out to be necessary and sufficient
(under the assumption $\EE|B_\tau|<\infty$).
\end{remark}

In Examples~\ref{ex:mst3}
and~\ref{ex:mst4} below,
$X$ will be a $d$-dimensional
$(\cF_t)$-Brownian motion
starting from $x_0\in\bbR^d$,
$d\geq2$,
on some filtered probability space.
The state space will be
$E:=\bbR^d\cup\{\infty\}$.
By $|\cdot|$ we denote
the Euclidean norm on $\bbR^d$.
It is well-known that, if $d>2$,
then $\lim_{t\to\infty}X_t=\infty$~a.s.,
while if $d=2$, then $X$ is recurrent.
Let us also recall that, for all $d\geq2$,
every one-point set in $\bbR^d$
is polar for~$X$.

\begin{example}[BM$^d$, $d>2$,
which is transient]
\label{ex:mst3}
Let $d>2$. Take $y\in\bbR^d$, $y\neq x_0$,
and set $g(x)=|x-y|^{2-d}$, $x\in E$.
By It\^o's formula,
$g(X)$ is a positive local martingale,
hence a closed supermartingale.
It has a strictly increasing quadratic variation,
hence no intervals of constancy.
Theorem~\ref{th:mst1} implies that
$$
\text{every }
(\cF_t)
\text{-stopping time }
\tau
\text{ is minimal for }
X.
$$
\end{example}

\begin{example}[BM$^2$,
which is recurrent]
\label{ex:mst4}
For $d=2$, due to recurrence of~$X$,
minimality is well-defined only for finite
$(\cF_t)$-stopping times.
Take $z\in\bbR^2$, $z\neq x_0$,
and set $g_z(x)=\log|x-z|$, $x\in E$.
Let us define the process
$Y_t=g_z(X_t)$, $t\geq0$.
By It\^o's formula and L\'evy's
characterisation theorem,
the process $Y$ satisfies
SDE~\eqref{eq:mst35}
with $\ka(x)=e^{-x}$ (and $y_0=g_z(x_0)$),
in particular, $Y$ is a local martingale.
Here, $\ka$ satisfies~\eqref{eq:mst40}
but not~\eqref{eq:mst50},
which means that $Y$ is not a martingale.
Denoting $X=(X^1,X^2)$ and $z=(z^1,z^2)$,
we have $Y=\log|X-z|\leq|X-z|\leq|X^1-z^1|+|X^2-z^2|$,
hence $\EE\sup_{s\leq t}Y^+_s<\infty$
for all $t\in[0,\infty)$.
Therefore, $Y$ is a submartingale.
Now Theorem~\ref{th:mst1}
with $g$ being $-g_z$
and Lemma~\ref{lem:mst1}
imply that any $(\cF_t)$-stopping time $\tau$
satisfying
\begin{equation}
\label{eq:mst80}
\left((\log|X_{t\wedge\tau}-z|)^+\right)_{t\geq0}
\text{ is uniformly integrable}
\end{equation}
is finite and minimal for~$X$
(again, finiteness of $\tau$ follows
from~\eqref{eq:mst80}
because the closed submartingale
$Y^\tau$ converges~a.s.).
Furthermore, Theorem~\ref{th:mst1}
with $g$ being $g_z$ implies
that any $(\cF_t)$-stopping time $\tau$
satisfying
\begin{equation}
\label{eq:mst90}
\EE\sup_{t\geq0}
(\log|X_{t\wedge\tau}-z|)^-
<\infty
\end{equation}
is finite and minimal for~$X$.
Summarising, for a two-dimensional
$(\cF_t)$-Brownian motion~$X$
starting from $x_0\in\bbR^2$,
any $(\cF_t)$-stopping time $\tau$
satisfying either~\eqref{eq:mst80}
or~\eqref{eq:mst90}
with some $z\in\bbR^2$, $z\neq x_0$,
is finite and minimal for~$X$.
\end{example}

\bigskip\noindent
\textbf{Acknowledgement.}
Alexander Gushchin
was partially supported
by the International La\-bo\-ra\-tory
of Quantitative Finance,
National Research University
Higher School of Economics,
Russian Federation Government grant,
N.~14.A12.31.0007.

%=======
%\nocite{*}
%\bibliographystyle{chicago}
%\bibliographystyle{amsplain}
%\bibliographystyle{plain}
\bibliographystyle{abbrv}
\bibliography{refs}

\begin{thebibliography}{10}

\bibitem{AcciaioBeiglbockPenknerSchachermayerTemme:13}
B.~Acciaio, M.~Beiglb{\"o}ck, F.~Penkner, W.~Schachermayer, and J.~Temme.
\newblock A trajectorial interpretation of {D}oob's martingale inequalities.
\newblock {\em Ann. Appl. Probab.}, 23(4):1494--1505, 2013.

\bibitem{AneGeman:00}
T.~An\'e and H.~Geman.
\newblock Order flow, transaction clock, and normality of asset returns.
\newblock {\em The Journal of Finance}, 55(5):2259--2284, 2000.

\bibitem{AnkirchnerHeyneImkeller:08}
S.~Ankirchner, G.~Heyne, and P.~Imkeller.
\newblock A {BSDE} approach to the {S}korokhod embedding problem for the
  {B}rownian motion with drift.
\newblock {\em Stoch. Dyn.}, 8(1):35--46, 2008.

\bibitem{AnkirchnerHobsonStrack:14}
S.~Ankirchner, D.~G. Hobson, and P.~Strack.
\newblock Finite, integrable and bounded time embeddings for diffusions.
\newblock {\em To appear in Bernoulli}, 2014.

\bibitem{AnkirchnerStrack:11}
S.~Ankirchner and P.~Strack.
\newblock Skorokhod embeddings in bounded time.
\newblock {\em Stoch. Dyn.}, 11(2-3):215--226, 2011.

\bibitem{Barndorff-NielsenShiryaev:10}
O.~E. Barndorff-Nielsen and A.~Shiryaev.
\newblock {\em Change of time and change of measure}.
\newblock Advanced Series on Statistical Science \& Applied Probability, 13.
  World Scientific Publishing Co. Pte. Ltd., Hackensack, NJ, 2010.

\bibitem{BaxterChacon:77}
J.~R. Baxter and R.~V. Chacon.
\newblock Enlargement of {$\sigma $}-algebras and compactness of time changes.
\newblock {\em Canad. J. Math.}, 29(5):1055--1065, 1977.

\bibitem{BeiglbockHenry-LaborderePenkner:13}
M.~Beiglb{\"o}ck, P.~Henry-Labord{\`e}re, and F.~Penkner.
\newblock Model-independent bounds for option prices---a mass transport
  approach.
\newblock {\em Finance Stoch.}, 17(3):477--501, 2013.

\bibitem{Birnbaum:61}
A.~Birnbaum.
\newblock On the foundations of statistical inference: binary experiments.
\newblock {\em Ann. Math. Statist.}, 32:414--435, 1961.

\bibitem{BrownHobsonRogers:01}
H.~Brown, D.~G. Hobson, and L.~C.~G. Rogers.
\newblock Robust hedging of barrier options.
\newblock {\em Math. Finance}, 11(3):285--314, 2001.

\bibitem{Cox:08}
A.~M.~G. Cox.
\newblock Extending {C}hacon-{W}alsh: minimality and generalised starting
  distributions.
\newblock In {\em S\'eminaire de probabilit\'es {XLI}}, volume 1934 of {\em
  Lecture Notes in Math.}, pages 233--264. Springer, Berlin, 2008.

\bibitem{CoxHobson:04}
A.~M.~G. Cox and D.~G. Hobson.
\newblock An optimal {S}korokhod embedding for diffusions.
\newblock {\em Stochastic Process. Appl.}, 111(1):17--39, 2004.

\bibitem{CoxHobson:06}
A.~M.~G. Cox and D.~G. Hobson.
\newblock Skorokhod embeddings, minimality and non-centred target
  distributions.
\newblock {\em Probab. Theory Related Fields}, 135(3):395--414, 2006.

\bibitem{CoxHobsonObloj:08}
A.~M.~G. Cox, D.~G. Hobson, and J.~Ob{\l}{\'o}j.
\newblock Pathwise inequalities for local time: applications to {S}korokhod
  embeddings and optimal stopping.
\newblock {\em Ann. Appl. Probab.}, 18(5):1870--1896, 2008.

\bibitem{CoxObloj:11a}
A.~M.~G. Cox and J.~Ob{\l}{\'o}j.
\newblock Robust hedging of double touch barrier options.
\newblock {\em SIAM J. Financial Math.}, 2(1):141--182, 2011.

\bibitem{CoxObloj:11}
A.~M.~G. Cox and J.~Ob{\l}{\'o}j.
\newblock Robust pricing and hedging of double no-touch options.
\newblock {\em Finance Stoch.}, 15(3):573--605, 2011.

\bibitem{CoxWang:13}
A.~M.~G. Cox and J.~Wang.
\newblock Root's barrier: construction, optimality and applications to variance
  options.
\newblock {\em Ann. Appl. Probab.}, 23(3):859--894, 2013.

\bibitem{DolinskySoner:14}
Y.~Dolinsky and H.~M. Soner.
\newblock Martingale optimal transport and robust hedging in continuous time.
\newblock {\em To appear in Probab. Theory Related Fields}, 2014.

\bibitem{GranditsFalkner:00}
P.~Grandits and N.~Falkner.
\newblock Embedding in {B}rownian motion with drift and the {A}z\'ema-{Y}or
  construction.
\newblock {\em Stochastic Process. Appl.}, 85(2):249--254, 2000.

\bibitem{Hall:69}
W.~J. Hall.
\newblock Embedding submartingales in {W}iener processes with drift, with
  applications to sequential analysis.
\newblock {\em J. Appl. Probability}, 6:612--632, 1969.

\bibitem{Hobson:98a}
D.~G. Hobson.
\newblock Robust hedging of the lookback option.
\newblock {\em Finance Stoch.}, 2(4):329--347, 1998.

\bibitem{Hobson:11}
D.~G. Hobson.
\newblock The {S}korokhod embedding problem and model-independent bounds for
  option prices.
\newblock In {\em Paris-{P}rinceton {L}ectures on {M}athematical {F}inance
  2010}, volume 2003 of {\em Lecture Notes in Math.}, pages 267--318. Springer,
  Berlin, 2011.

\bibitem{KaratzasKardaras:07}
I.~Karatzas and C.~Kardaras.
\newblock The num\'eraire portfolio in semimartingale financial models.
\newblock {\em Finance Stoch.}, 11(4):447--493, 2007.

\bibitem{KaratzasShreve:91}
I.~Karatzas and S.~E. Shreve.
\newblock {\em Brownian {M}otion and {S}tochastic {C}alculus}, volume 113 of
  {\em Graduate Texts in Mathematics}.
\newblock Springer-Verlag, New York, second edition, 1991.

\bibitem{Kotani:06}
S.~Kotani.
\newblock On a condition that one-dimensional diffusion processes are
  martingales.
\newblock In {\em In memoriam {P}aul-{A}ndr\'e {M}eyer: {S}\'eminaire de
  {P}robabilit\'es {XXXIX}}, volume 1874 of {\em Lecture Notes in Math.}, pages
  149--156. Springer, Berlin, 2006.

\bibitem{KramkovSchachermayer:99}
D.~Kramkov and W.~Schachermayer.
\newblock The asymptotic elasticity of utility functions and optimal investment
  in incomplete markets.
\newblock {\em Ann. Appl. Probab.}, 9(3):904--950, 1999.

\bibitem{Monroe:72}
I.~Monroe.
\newblock On embedding right continuous martingales in {B}rownian motion.
\newblock {\em Ann. Math. Statist.}, 43:1293--1311, 1972.

\bibitem{Monroe:78}
I.~Monroe.
\newblock Processes that can be embedded in {B}rownian motion.
\newblock {\em Ann. Probability}, 6(1):42--56, 1978.

\bibitem{Obloj:04}
J.~Ob{\l}{\'o}j.
\newblock The {S}korokhod embedding problem and its offspring.
\newblock {\em Probab. Surv.}, 1:321--390, 2004.

\bibitem{Obloj:12}
J.~Ob{\l}{\'o}j.
\newblock On some aspects of the {S}korokhod {E}mbedding {P}roblem and its
  applications in {M}athematical {F}inance.
\newblock {\em Notes for the students of the 5th European Summer School in
  Financial Mathematics}, 2012.

\bibitem{OblojUlmer:12}
J.~Ob{\l}{\'o}j and F.~Ulmer.
\newblock Performance of robust hedges for digital double barrier options.
\newblock {\em Int. J. Theor. Appl. Finance}, 15(1):1250003, 34, 2012.

\bibitem{PedersenPeskir:01}
J.~L. Pedersen and G.~Peskir.
\newblock The {A}z\'ema-{Y}or embedding in non-singular diffusions.
\newblock {\em Stochastic Process. Appl.}, 96(2):305--312, 2001.

\bibitem{Peskir:00}
G.~Peskir.
\newblock The {A}z\'ema-{Y}or embedding in {B}rownian motion with drift.
\newblock In {\em High dimensional probability, {II} ({S}eattle, {WA}, 1999)},
  volume~47 of {\em Progr. Probab.}, pages 207--221. Birkh\"auser Boston,
  Boston, MA, 2000.

\bibitem{RevuzYor:99}
D.~Revuz and M.~Yor.
\newblock {\em Continuous {M}artingales and {B}rownian {M}otion}, volume 293 of
  {\em Grundlehren der Mathematischen Wissenschaften}.
\newblock Springer-Verlag, Berlin, third edition, 1999.

\bibitem{Skorokhod:61}
A.~V. Skorohod.
\newblock {\em Issledovaniya po teorii sluchainykh protsessov
  ({S}tokhasticheskie differentsialnye uravneniya i predelnye teoremy dlya
  protsessov {M}arkova)}.
\newblock Izdat. Kiev. Univ., Kiev, 1961.

\bibitem{Skorokhod:65}
A.~V. Skorokhod.
\newblock {\em Studies in the theory of random processes}.
\newblock Translated from the Russian by Scripta Technica, Inc. Addison-Wesley
  Publishing Co., Inc., Reading, Mass., 1965.

\end{thebibliography}
\end{document}